\documentclass{article}
\usepackage[english]{babel}
\usepackage[utf8]{inputenc}
\usepackage{amsmath}
\usepackage{graphicx}
\usepackage{amsfonts}
\usepackage{enumitem}
\usepackage{amssymb}
\usepackage[colorinlistoftodos]{todonotes}
\usepackage{geometry}
\usepackage{amsthm}
\usepackage{enumitem} 
\usepackage{hyperref}

\newcommand{\remove}[1]{}
\usepackage{tikz}

\long\def\onefigure#1#2{
\begin{figure*}[tbp]
\begin{center}
#1
\end{center}
\caption{#2}
\end{figure*}
} 

\newcommand{\lipefig}[2]  
{\onefigure{\mbox{\psfig{file=#1.eps}}}{\label{f:#1} #2} }

\newcommand{\cH}{\mathcal{H}}

\newcommand{\si}{\sigma}

\newcommand{\R}{\mathbb{R}}

\newcommand{\cL}{\mathcal{L}}

\newtheorem{thm}{Theorem}[section]
\newtheorem{lem}[thm]{Lemma}

\theoremstyle{definition}

\DeclareMathOperator{\CR}{cr}
\DeclareMathOperator{\conv}{conv}

\DeclareMathOperator{\disc}{disc}

\title{Line arrangements, bounded cells, and discrepancy}
\author{Imre B\'ar\'any}
\date{September 2026}
\begin{document}

\maketitle
\textbf{Abstract}: An arrangement of $n$ lines in the plane defines cells, i.e., connected components of the complement of the arrangement. Answering a question of Volkmar Welker we show that, apart from two special cases, there is always a line $L \in \cL$ such that one of the open halfplanes whose boundary is $L$ contains more than half of the bounded cells of the arrangement. We also consider some unusual discrepancy problems that are related to this question. 



\section{Introduction}
Let $\cL=\{L_1,\ldots,L_n\}$ be an arrangement of $n$ distinct lines in $\R^2$. For $L\in \cL$, let $L^+$ and $L^-$ denote the open halfplane above and below (respectively) of $L$. Here we assume that $L$ contains no vertical line. The arrangement defines {\sl cells} in $\R^2$: a cell is a relatively open (and convex) subset $C$ of $\R^2$  such that for every $L\in \cL$ and for every pair $x,y\in C$, $x,y\in L^+$ or $x,y\in L^-$ or $x,y\in L$. The empty set is not a cell by definition. We write $B=B(\cL)$ for the set of all bounded cells of the arrangement. For a set $M\subset \R^2$, $b(M)$ denotes the number of bounded cells in $M$. 

\smallskip
Motivated by a problem in algebraic combinatorics, Volkmar Welker \cite{Wel} asked the following question. Is it true that there is always a line $L\in \cL$ with $\max \{b(L^+), b(L^-)\} > \frac {|B|}2$? 

\smallskip
Consider three simple cases, see Figure~\ref{fig:abc} below:
\begin{figure}[h!]
\centering
\includegraphics[scale=0.8]{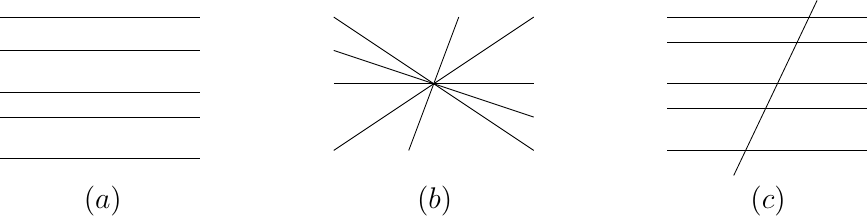}
\caption{Three special cases.}
\label{fig:abc}
\end{figure}
\begin{enumerate}
\item[(a)] when $B=\emptyset$. Then all lines in $\cL$ are parallel and $b(L^+)=b(L^-)=0$ and $|B|=0$,

\item[(b)] when $B$ is a single point,  then all lines in $\cL$ contain this points and $b(L^+)=b(L^-)=0$ and $|B|=1$,

\item[(c)] when $\cL$ consists of $n-1$ parallel lines, and another, non-parallel line, then $|B|=2n-3$ and $\max \{b(L^+), b(L^-)\}=2(n-2).$
\end{enumerate}
The answer to Welker's problem is clearly false in cases (a) and (b) and is true in case (c). From now on we assume that $n\ge 3$. 
Our main result gives a positive answer to Welker's question.

\begin{thm}\label{th:main} Apart from cases (a) and (b) there is always an $L\in \cL$ satisfying 
\begin{equation}\label{eq:Wel}
\max \{b(L^+), b(L^-)\} > \frac {|B|}2. 
\end{equation}
\end{thm}

This inequality is related to discrepancy. The {\sl bounded cell discrepancy} of a line $L\in \cL$, and the {\sl bounded cell discrepancy} of the arrangement $\cL$ are defined as 
\[
\disc_{BC}(L)=|b(L^+)-b(L^-)| \mbox{ and }\disc_{BC}(\cL)=\max_{L\in \cL}\disc_{BC}(L), \mbox{ respectively}.
\]
We define, as usual, $\disc_{BC}(n)=\min \disc_{BC}(\cL)$ where the minimum is taken over all $n$-line arrangements $\cL$ avoiding cases (a) and (b).

\begin{thm}\label{th:BCdisc} For every $n\ge 3$ we have
\[
\max \left(1,\frac 13(n-14)\right) \le \disc_{BC}(n) \le 2(n-2).
\]
\end{thm}

Thus the order of magnitude of $\disc_{BC}(n)$ is linear in $n$. Example (c) gives the upper bound. We prove the lower bound in the last section.

\section{Crossing discrepancy}

A {\sl crossing} (or a vertex) in an arrangement $\cL$ is the intersection point of two lines in $\cL$. Let $X=X(\cL)$ be the set of crossings. For a set $M\subset \R^2$ define $v(M)=|X \cap M|$. The {\sl crossing discrepancy} of $L \in \cL$ is $\disc_X(L)=|v(L^+)-v(L^-)$ and, as one should expect, the {\sl crossing discrepancy} of $\cL$ is
\[
\disc_X(\cL)=\max_{L\in \cL}\disc_X(L) \mbox{ and } \disc_X(n)=\min \disc_X(\cL)
\]
where the minimum is taken over all $n$-line arrangements except cases (a) and (b). In (c) the crossing discrepancy is exactly $n-2$. In geometry the crossing discrepancy seems to be more natural than the bounded cell discrepancy.  Our next result shows that $\disc_X(n)$ grows linearly with $n$.

\begin{thm}\label{th:Xdisc} 
\[
\max \left\{1, \frac 16(n-14)\right\}\le \disc_X(n) \le n-2.
\]
\end{thm}

In general one expects $\disc_X(\cL)$ to be of order $|X|$, otherwise every line in $\cL$ is an ``almost halving" line for $X$. Of course $\disc_X(\cL) < |X|$. We conjecture that $\disc_X(\cL) > c|X|$ for some universal constant $c>0$. We prove this in a special case, namely for general position arrangements, meaning that there are no parallel lines in $\cL$ and every $x\in X$ is the intersection point of a unique pair of lines in $\cL$. 

\begin{thm}\label{th:gen} For a general position $n$-line arrangement $\cL$ 
\[
\disc_X(\cL) \ge \frac {n^2}{\sqrt {60}}\sqrt{1+O(n^{-1})}.
\]
\end{thm}

The following example shows an arrangement with $\disc_X(\cL)<n^2/4$. Assume $n$ is odd, consider a regular $n$-gon inscribed in the unit disk. Let $\cL$ consist of the lines that are tangent to the unit disk at the vertices of this $n$-gon. Note that this is a very uniform arrangement, because by symmetry the value of $\disc_X(L)$ is the same for every $L \in \cL$. To determine this value assume that $(0,1)\in \R^2$ is a vertex of the $n$-gon, and $L\in \cL$ is tangent to the unit disk at $(0,1)$. Then $\disc_X(L)$ is equal to the number of crossings strictly between the lines $y=1$ and $y=-1$. The two lines giving such a crossing have be tangent to the unit disk at two vertices of the $n$-gon on the same side of the $y$ axis. The number of such pair of lines is
\[
2{\frac{n-1}2 \choose 2}=\frac{n^2-4n+3}4.
\]
For this nice and simple argument I'm indebted to Pavel Valtr \cite{Val}.

\smallskip
Welker's question is in fact more general. Namely, assume that $\cH=\{H_1,\ldots,H_n\}$ is an arrangement of $n$ distinct hyperplanes in $\R^d$.  For $H\in \cH$, let $H^+$ and $H^-$ denote the open halfspace above and below (respectively) of $H$. The arrangement defines {\sl cells} in $\R^d$: a cell is a relatively open (and convex) subset $C$ of $\R^d$ defined the same way as in the planar case. The original question was whether there is always an $H \in \cH$ such that either $H^+$ or $H^-$ contains at least half of the bounded cells. This question remains open for all $n\ge 3$. The bounded cell and crossing  discrepancies can be extended to higher dimension of course.

\section{Proof of Theorem~\ref{th:main}}
Assume $x\in X$ is a vertex of the convex hull of $X$. With a suitable affine transformation we can achieve that $x$ coincides with the origin, that is $x=0$, and we can choose the two lines $L_1,L_2\in \cL$ with $L_1\cap L_2=0$ so that the slope of $L_1$ is negative, the slope of $L_2$ is positive and the slope of every other line $L \in \cL$ with $0\in L$ is between the slopes of $L_1$ and $L_2$.

\smallskip
Assume $|X\cap L_i|=k_i$ for $i=1,2$, of course $k_1,k_2\ge 1$, see Figure~\ref{fig:Mcone}. The cases $k_1=k_2=1$ or $=0$ are excluded. When $k_1=1$ and $k_2\ge 2$, then all lines $L\in \cL$ distinct from $L_1,L_2$ are either parallel with $L_1$ or contain the origin. Assume $s$ of them contain the origin. When $s=0$ we are in case (c) which is easy: $|B|=2k_2-1$, $b(L_1^-)=0,b(L_1^+)=2k_2-2$ and 
\[
b(L_1^+)=2k_2-2>\frac {|B|}2=k_2-\frac 12
\]
follows since $k_2\ge 2$.

\smallskip
 \begin{figure}[h!]
\centering
\includegraphics[scale=0.65]{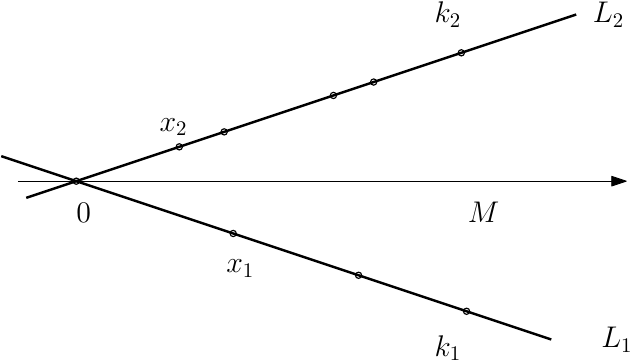}
\caption{Lines $L_1,L_2$, the cone $M$.}
\label{fig:Mcone}
\end{figure}

We can assume now that either $k_1=1$ and $s\ge1$ or both $k_1,k_2\ge 2$. As $b(L_i)=2k_i-1$, we have $b(L_i^+)+b(L_i^-)+(2k_i-1)=|B|$ for $i=1,2$, implying that the midpoint of the segment $[b(L_i^+),b(L_i^-)]$ is at distance $k_i-1/2$ from $|B|/2$. So inequality (\ref{eq:Wel}) would follow if $|b(L_i^+)-b(L_i^-)|>2k_i$ for either $i=1$ or $i=2$. Assume that on the contrary, $\disc_{BC}(L_i)=|b(L_i^+)-b(L_i^-)|\le 2k_i$ for both $i=1,2$. 

\smallskip
Let $M$ denote the open cone between the lines $L_1$ and $L_2$ containing the positive half of the $x$-axis, see Figure~\ref{fig:Mcone}. Then 
\[
b(L_2^+)+(2k_2-2)+b(M)=b(L_1^+) \mbox{ and similarly }b(L_1^-)+(2k_1-2)+b(M)=b(L_2^-).
\]
Consequently 
\begin{eqnarray*}
b(L_1^+)-b(L_2^+)+b(L_2^-)-b(L_1^-)&=&[b(L_1^+)-b(L_1^-)]-[b(L_2^+)-b(L_2^-)]\\
&=&2(k_1+k_2+b(M)-2)
\end{eqnarray*}
implying that 
\begin{equation}\label{eq:dir}
\disc_{BC}L_1+\disc_{BC}L_2 \ge 2(k_1+k_2+b(M)-2).
\end{equation}

\begin{lem} If $X \cap M \ne \emptyset$, then $b(M)\ge 3$.
\end{lem}

{\bf Proof.} Assume $x \in X\cap M$. Then each one of the two lines whose intersection is $x$ contains a bounded segment (with endpoint $x$). These two segments and $x$ form 3 bounded cells in $M$ and $b(M)\ge 3$ indeed. One can show further that $b(M)\ge 4$ but this is not needed.\qed

\smallskip
We are going to show that $X \cap M \ne \emptyset$ apart from three special cases that we will treat separately. This will finish the proof of Theorem~\ref{th:main} because then $b(M)\ge 3$ and the right hand side of (\ref{eq:dir}) is at least $2(k_1+k_2+1)$ while the left hand side is at most  $\disc_{BC}(L_1)+\disc_{BC}(L_2)\le 2(k_1+k_2)$.

\smallskip
If $k_1=1$ and $s\ge 1$ then there is a line $L_3\in \cL$ parallel with $L_1$, and another one, say $L_4$, (distinct from $L_1,L_2$), containing the origin. Their intersection is a crossing in $M$.

\smallskip
We are left with the case $k_1,k_2\ge 2$. For $i=1,2$ let $x_i$ be the crossing on $L_i$ closest to the origin. Write $H$ for the line containing $x_1$ and $x_2$. If $H \notin \cL$, then there is a line in $\cL$ containing $x_1$ and distinct fron $L_1$ and another line containing $x_2$ and distinct fron $L_2$. The intersection of these two lines is a crossing in $M$ again.

\smallskip
So $H \in\cL$ and assume $s\ge 1$. Then a line in $\cL$ containing the origin (and distinct from $L_1,L_2$) intersects $H$ in a crossing in $M$. So $s=0$ and  $M$ contains the bounded cell $(x_1,x_2)$ and the triangle $\conv\{0,x_1,x_2\}$ showing that $b(M)\ge 2$. When $n=3$ we simply check that $L_1^+=4$ is indeed larger than $|B|/2=7/2$. This is one of the special cases.  

\smallskip
So $n\ge 4$ and $s=0$ and there is a line $L\in \cL$ different from $L_1,L_2,H$. If $L\cap M$ is a segment, then it is a bounded cell in $M$ and $b(M)\ge 3$. This is the second special case.

\smallskip
So  $L\cap M$ is not a segment and then $L$ is parallel with $L_1$ or with $L_2$. If there is one line parallel with $L_1$ and another one parallel with $L_2$, then their intersection is a crossing in $M$ and we are done. So each line $L\in \cL$ different from $L_1,L_2,H$ is parallel with $L_2$, say. Then $k_1=2$ and a simple checking gives that $b(L_1^+)=2k_2$ which is larger than half of $|B|=2k_2+3$. This was the last special case.\qed

\section{Proof of Theorem~\ref{th:Xdisc}}

 Set $v(M)=m$ and $k_1+k_2=k$ and we have, the same way as in the proof of Theorem~\ref{th:main}, that
\[
v(L_2^+)+(k_2-1) +m=v(L_1^+) \mbox{ and } v(L_1^-)+(k_1-1)+m=v(L_2^-) 
\]
implying that 
\begin{equation}\label{eq:Xdisc}
\disc_X(L_1)+\disc_X(L_2)\ge k+2(m-1).
\end{equation}

As $k\ge 3$, this implies that $\disc_X(L_1)$ or $\disc_X(L_2)$ is at least one.

\smallskip
We are going to use the crossing lemma, originally due to Ajtai et al. \cite{ACNS}. It is about a graph $G$ with $v$ vertices and $e$ edges drawn in the plane. The {\sl crossing number} $\CR(G)$ of $G$ is the number of points where two edges intersect in a non-vertex of $G$.  We use it in the form due to Pach and T\'oth \cite{PachT}, the one in Ackerman \cite{Ack} would give the same order of magnitude for $\disc_X(n)$ with a slightly weaker constant.

\begin{lem} If under the previous conditions $e>4v$ then 
\[
\CR(G) > \frac 1{64} \frac {e^3}{v^2}.
\]
\end{lem}
 \begin{figure}[h!]
\centering
\includegraphics[scale=0.65]{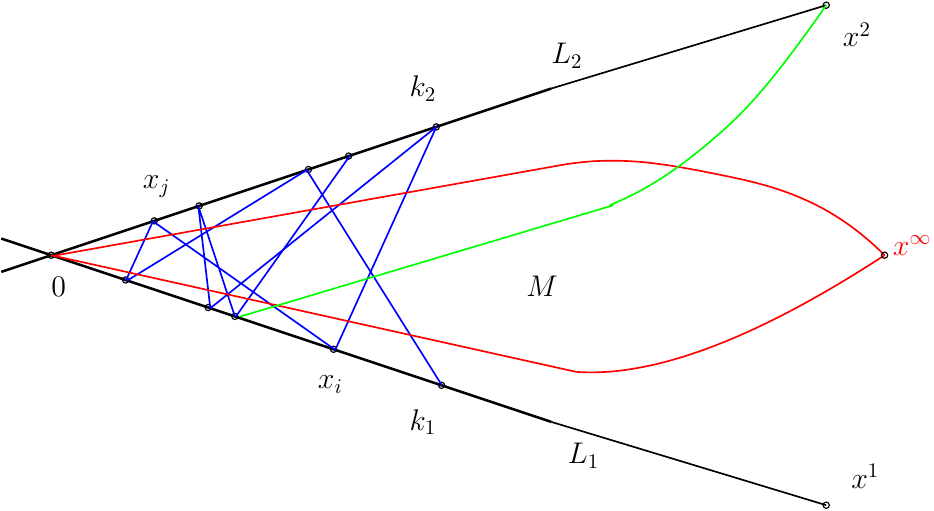}
\caption{The graph $G$.}
\label{fig:Graph}
\end{figure}
In Figure~\ref{fig:Graph} there is a graph $G$ drawn in the plane. We explain next what this graph is and how it is drawn. The vertices of $G$ are the points in $X\cap (L_1\cup L_2)$ plus two more vertices $x^1$ and $x^2$ and possibly one more, namely $x^{\infty}$. Here $x_i$ is at infinity on $L_i $, $i=1,2$ and $x^{\infty}$ is at infinity on the positive half of the $x$ axis but only if there is a line in $\cL$ containing the origin, different from $L_1,L_2$. So $G$ has $v=k+t$ vertices where $t\in \{1,2\}$. 

\smallskip
If the line through $x_i \in X\cap L_1$ and $x_j \in X \cap L_2$ is in $\cL$, then the segment $[x_i,x_j]$ is an edge of $G$. This edges are coloured blue on Figure~\ref{fig:Graph}. The last crossing on $L_i$ and $x^i$ form an edge in $G$ for $i=1,2$. A line in $\cL$ parallel with $L_2$ and containing $x \in L_1$ defines an edge connecting $x$ and $x^2$ as shown in green colour on the figure. Similarly a line in $\cL$ parallel with $L_1$ and containing $x \in L_2$ defines an edge connecting $x$ and $x^1$. If there is a line $L\in \cL$ containing the origin and distinct from $L_1,L_2$, then there is an edge in $G$ connecting the origin and $x^{\infty}$ along $L$. This is shown, twice, in the figure in colour red. In the drawing of $G$ the red and the green edges follow their line for long and then are bended in order to reach $x^i$ or $x^{\infty}$. Two consecutive crossings on $L_1$ and on $L_2$ also form an edge of $G$. This is $k-2$ edges. Thus $G$ has $e=k+n$ edges. 

\smallskip 
Suppose that $e=k+n> 4v\ge 4(k+t)$. The crossing lemma applies and shows that
\[
m=\CR(G)>\frac 1{64} \frac {(k+n)^3}{(k+t)^2}>\frac 1{64}\left(1+\frac{n-t}{k+t}\right)^2(k+n)\ge \frac 14(k+n)
\]
since $k+n>4(k+t)$ implies that $\left(1+\frac{n-t}{k+t}\right)>4$. We conclude that $m>(k+n)/4$ and it follows from (\ref{eq:Xdisc}) that  $\disc_X(L_i)\ge \frac 14 n$ for $i=1$ or 2. 

\smallskip
Assume next that  $e=k+n\le 4v= 4(k+t)$. Then $k\ge \frac 13(n-4t)\ge \frac 13(n-8)$. Using (\ref{eq:Xdisc}) again we see that $\disc_X(L_i)\ge \frac 16 (n-14)$ for $i=1$ or 2.\qed

\section{Proof of Theorem~\ref{th:gen}}

Order the lines in $\cL$ linearly by the relation $\prec$ so $L_1\prec L_2$ means that $L_1$ comes before $L_2$ in this order (and $L_1\ne L_2)$.
Assume $H,K_1,K_2 \in \cL$, $H\ne K_1,K_2$ and $K_1\prec K_2$. Let $x=K_1\cap K_2$ which is unique because $\cL$ is in general position and $x\notin H$. Define $\si(H,K_1,K_2)=1$ if $x\in H^+$ and $-1$ if $x\in H^-$. With this notation $\disc_X H=\left|\sum_{K_1,K_2}\si(H,K_1,K_2)\right|$ with summation taken over all pairs $K_1,K_2$ distinct from $H$ and $K_1\prec K_2$. 

The key step in the proof is 
\begin{lem}\label{l:6typ} $\sum_{H \in \cL}(\disc_X H)^2 \ge 2{n \choose 5}+O(n^4)$.
\end{lem}

{\bf Proof.} We expand the sum of the squares as
\begin{eqnarray*}
\sum_{H \in \cL}(\disc_X H)^2&=&\sum_{H \in \cL}\sum_{K_1,K_2,L_1,L_2}\si(H,K_1,K_2)\si(H,L_1,L_2)\\
   &=&\sum_{H \in \cL}\sum_{K_1,K_2,L_1,L_2}\si(H,K_1,K_2)\si(H,L_1,L_2)+O(n^4). 
\end{eqnarray*}
where in the second sum $H,K_1,K_2$ are distinct and $K_1\prec K_2$ and $H,L_1,L_2$ are distinct and $L_1\prec L_2$ while in the third sum $H,K_1,K_2,L_1,L_2$ are all distinct and $K_1\prec K_2$, $L_1\prec L_2$ and the term $O(n^4)$ accounts for the terms where $H,K_1,K_2,L_1,L_2$ are not all distinct. 

For every 5-tuple $H_1,H_2,\ldots,H_5$ (of distinct lines in $\cL$) we collect the terms in the last sum containing exactly these five lines. Such a term is of the form $\si(H_i,H_x,H_y)\si(H_i,H_u,H_v)$ where $\{i,x,y,u,v\}=\{1,2,3,4,5\}$ and $H_x\prec H_y$ and $H_u\prec H_v$. As both $\si(H,K_1,K_2)\si(H,L_1,L_2)$ and $\si(H,L_1,L_2)\si(H,K_1,K_2)$ appear for a given 5-tuple, we assume $H_x\prec H_u$ and will have to multiply the outcome by 2. We sum these terms for a fixed $i$. There are three such terms, each equal to $\pm 1$. So this sum is $\pm 3 $ or $\pm 1$. Then we add these sums for the five possible choices of $H_i$ and denote the outcome by $\si(H_1,\ldots,H_5)$.
\begin{figure}[h!]
\centering
\includegraphics[scale=0.8]{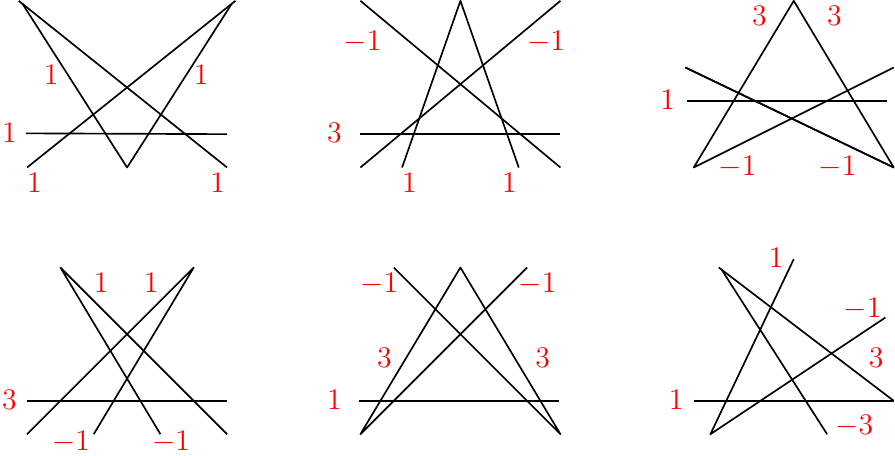}
\caption{The six combinatorially different types of 5-lines.}
\label{fig:6typ}
\end{figure}

There are exactly six combinatorially distinct types of 5-tuples of lines in $\R^2$ in general position; see Ringel \cite{Rin}, page 90, and Cutler et al. \cite{CKS}, Figure 9, and also \cite{Christ}. The 6 types are shown on Figure~\ref{fig:6typ}. Each five tuple $H_1,\ldots,H_5$ from our $\cL$ is one of them. The computation of $\si(H_1,\ldots,H_5)$ for one combinatorial type starts with a three term sum for each choice $H_i$. This sum is written (in red) next to the actual line in Figure~\ref{fig:6typ}. All of a sudden a miracle happens: it turns out that $\si(H_1,\ldots,H_5)\ge 1$ for every combinatorial type. One can see that $\si(H_1,\ldots,H_5)$ equals 5,3,5 in the top row, left to right, of Figure~\ref{fig:Graph} and 3,5,1 in the bottom row.\qed

\medskip
It is easy to complete the proof from here. There are $n$ terms in the sum $\sum_{H \in \cL}(\disc_X H)^2$ so one of them is at least $\frac 2n \left({n \choose 5}+O(n^4)\right).$ Then for some $H \in \cL$
\[
\disc_X H \ge \sqrt{\frac 2n \left({n \choose 5}+O(n^4)\right)}\ge \frac {n^2}{\sqrt {60}}\sqrt{1+O(n^{-1})}.\qed
\]

\section{Proof of Theorem~\ref{th:BCdisc}}

We note that $\disc_{BC}(n)\ge 1$ follows from $\disc_X (n)\ge 1$, the lower bound in Theorem~\ref{th:Xdisc}. To fix notation let $\ell_i\ge 0$ be the number of lines in $\cL$ parallel with $L_i$ (excluding $L_i$), $i=1,2$, and $s\ge 0$ be the number of lines in $\cL$ containing the origin excluding $L_1,L_2$. We can't have $k_1=k_2=1$. 

\smallskip
Assume first that $k_1=1, k_2\ge 2$. If $s=0$, then $\cL$ is case (c) and $\disc_{BC}(L_1)=2(n-2)$. If $s\ge 1$, then $k_2=n-s$. Each line passing through the origin (including $L_2$ but not $L_1$) contains $2(k_2-1)$ bounded cells. There are further $s(k_2-1)$ bounded 2-dimensional cells in $L_1^+$. Thus $\disc_{BC}(L_1)=b(L_1^+)-b(L_1^-) \ge 2(s+1)(k_2-1)+s(k_2-1)=(3s+2)(n-1-s)$. It is easy to see that this expression is larger than $2(n-2)$ for all $s=1,2,\ldots,n-2$. 

\smallskip
We use inequality (\ref{eq:dir}) in order to give a lower bound on $k+b(M)-2$ where $k=k_1+k_2$. We will work again with the graph $G$ from the previous section. $G$ has $e=k+n$ edges and $v=k+t$ vertices where $t \in \{1,2\}$. If $e\le 4v$ then $k\ge \frac 13(n-8)$ and $k+b(M)-2\ge k-2\ge  \frac 13(n-14)$ implying that $\min_{i=1,2} \disc_{BC}L_i\ge \frac 13(n-14)$. 

\smallskip
We apply the crossing lemma when $e> 4v$ which gives, as we have seen, that $M$ contains at least $m=\frac 14(k+n)$ vertices.  Each such vertex $x$ is the intersection of two lines from $\cL$. Each one of these lines contains a bounded 1-dimensional cell whose endpoint is $x$. This gives $3m$ bounded cells in $M$. Thus $\min_{i=1,2} \disc_{BC}L_i\ge \frac 34(n+k)\ge \frac 34(n+1)$ in this case. \qed

\smallskip
We remark that, quite possibly, $\disc_{BC} (n)=2(n-2)$.

\medskip
{\bf Acknowledgements.} I'm indebted to Volkmar Welker for the question that initiated this piece of research. I also thank H\'el\`ene Barcelo, Patty Commins, and John Shareshian for the many useful and pleasant discussions we had about discrepancies and related issues. This piece of research is partially based upon work supported by the National Science Foundation under Grant No. DMS-2424139, while the author was in residence at the Simons Laufer Mathematical Sciences Institute in Berkeley, California, during the Spring 2026 semester.

\vskip0.3cm

\noindent
Imre B\'ar\'any \\
Alfr\'ed R\'enyi Institute of Mathematics, HUN-REN\\
13 Re\'altanoda Street, Budapest 1053 Hungary\\
{\tt barany.imre@renyi.hu}, and\\
Department of Mathematics, University College London\\
Gower Street, London, WC1E 6BT, UK\\

\end{document}